\documentclass{article}

\usepackage{amsmath}
\usepackage{amsthm}
\usepackage{amsfonts}
\usepackage{pgfplots}
\usepackage{amssymb}
\usepackage[utf8]{inputenc}
\usepackage{tikz}
\usetikzlibrary{positioning}
\usetikzlibrary{matrix,arrows,decorations.pathmorphing}
\usetikzlibrary{patterns}
\usetikzlibrary{through,calc,3d}
\usetikzlibrary{plotmarks}
\usetikzlibrary{pgfplots.groupplots}
\usepackage{varioref}

\usepackage{microtype}\usepackage{mathtools}%for figure placements

\usepackage{subfigure}                  % for sub-figures
\usepgfplotslibrary{external}
\newtheorem{definition}{\bf Definition}[section]
\newtheorem{theorem}{\bf Theorem}[section]
\newtheorem{lemma}{\bf Lemma}[section]

\newcommand{\inner}[2]{\ensuremath{\big< #1, #2 \big>}}
\newcommand{\id}{\mathrm d}

\usepackage[normalem]{ulem}

\begin{document}

%%%% Article title to be placed here
\title{Reduced-order Description of Transient Instabilities and Computation of Finite-Time Lyapunov Exponents}
\author{Hessam Babaee$^{1}$, Mohamad Farazmand$^1$,  George Haller$^2$,\\ Themistoklis P. Sapsis$^1$\thanks{Corresponding
author: {sapsis@mit.edu},
        Tel: (617) 324-7508, Fax: (617) 253-8689%
        }\\
$^{1}$Department of Mechanical Engineering, MIT\\
$^{2}$Department of Mechanical and Process Engineering, ETH Zurich
 }
\date{\today}

%%%% Keyword entries to be placed here %%%%
%\keywords{dimension reduction, finite-time instability}

\maketitle
%%%% Abstract text to be placed here %%%%%%%%%%%%
\begin{abstract}
High-dimensional chaotic dynamical systems can exhibit strongly transient features. These are often associated with  instabilities that have finite-time duration. Because of the finite-time character of these transient events, their detection through
infinite-time methods, e.g. long term averages, Lyapunov exponents or information about the statistical steady-state, is 
not possible. Here we utilize a recently developed framework, the Optimally Time-Dependent (OTD) modes, to extract a time-dependent subspace that spans the modes associated with transient features  associated with finite-time instabilities. As the main result, we prove that the OTD modes, under appropriate conditions, converge exponentially fast to the eigendirections of the Cauchy--Green tensor associated with the most intense finite-time instabilities. Based on this observation, we develop a reduced-order method for the computation of finite-time Lyapunov exponents (FTLE) and vectors. In high-dimensional systems, the computational cost of the reduced-order method is orders of magnitude lower than the full FTLE computation. We demonstrate the validity of the theoretical findings
on two numerical examples.

\end{abstract}
%%%%%%%%%%%%%%%%%%%%%%%%%%%

%%%%%%%%%%%%%%% End of first page %%%%%%%%%%%%%%%%%%%%%
\section{Introduction}
 
For a plethora of dynamical systems transient phenomena --- usually associated with finite-time instabilities --- define the most important aspects of the response. Examples include chaotic fluid systems \cite{chandler13, Farazmand2016}, nonlinear waves \cite{cousins_sapsis, cousinsSapsis2015_JFM} and networks \cite{Susuki2012a, Cornelius2013}. For the analysis of such systems and the formulation of prediction and control algorithms it is critical to identify modes associated with these strongly transient features. A first important challenge towards identifying these modes is the high-dimensionality of the response. This high-dimensionality does not necessarily imply that the transient features are also high-dimensional. In fact it is often the case that the modes associated with transient features are very few. However, because of the broad spectrum of the response it is hard to identify them using energy-based criteria. An additional challenge is associated with the intrinsically time-dependent character of these features which is, in general, non-periodic and makes it essential to consider arbitrary time-dependence for these modes as well. 

In the context of dynamical systems numerous approaches have been developed to characterize the transient features associated with non-periodic behavior. The notion of Lagrangian Coherent Structures (LCS) emerged from the fluid dynamics problem of mixing as a general method to characterize dynamical systems with arbitrary time-dependence \cite{haller00, Haller01a, Lekien07}. The method quantifies finite-time instabilities through the explicit computation of the maximum eigenvalue of the Cauchy--Green tensor on every location of the phase space. For low-dimensional systems it provides a complete and unambiguous description of the finite-time instabilities by identifying parts of the phase space associated with such responses. However, given the fast growth of the computational cost with respect to the phase space dimensionality it is not possible to apply it for very high- or infinite-dimensional systems in order to characterize finite-time instabilities. 

For such problems, other approaches focusing more on the description of the system attractor rather than the whole phase space have also been developed in the context of uncertainty quantification and model order reduction. Specifically, the notion of dynamical orthogonality \cite{SapsisLermusiaux09, SapsisLermusiaux10, Hou13a, Hou13b} allows for the computation of time-dependent, dynamically orthogonal (DO) modes that lead to reduced-order description of chaotic attractors with low-intrinsic dimensionality but with strong time-dependence \cite{sapsis11a, sapsisdijkstra}. The same notion can be combined with stochastic closures \cite{sapsis_majda_mqg, sapsis_majda_tur} in order to quantify statistics in time-dependent subspaces for turbulent dynamics systems with very broad spectra (or high intrinsic dimensionality) and strongly transient features \cite{sapsis_majda_mqgdo, Majda2014}. In all of these studies the effectiveness of the derived DO modes on capturing transient responses was demonstrated numerically and the very few theoretical studies related to this problem were constrained in very specific setups. One such case is the study of linear parabolic (stable) equations in \cite{nobile15} where it was shown that the DO modes tend to capture the most energetic directions. However, there is no general result indicating where the DO modes converge for an arbitrary (nonlinear) dynamical system.

To circumvent this limitation a minimization principle was introduced in \cite{Babaee} that allows for the derivation of equations that evolve a time-dependent set of modes, the Optimally-Time Dependent (OTD) modes, along a given trajectory of the system. These modes are identical with those obtained from the DO equations in the deterministic limit, i.e. for the case where the stochastic energy goes to zero. To this end, the OTD modes can be seen as the deterministic analog of DO modes. For sufficiently long times where the system reaches an equilibrium it was proven in \cite{Babaee} that the OTD modes converge to the most unstable directions of the system in the asymptotic limit. Moreover, the same modes were utilized in \cite{Farazmand2016} to formulate predictors of extreme events for Navier-Stokes equations and nonlinear waves problems. 

Despite their success on numerically capturing transient phenomena associated with finite-time instabilities the exact relation between OTD modes (or DO modes) and finite-time instabilities remains an open problem. In this work we provide a rigorous link between the OTD modes and finite-time instabilities, by showing that under general conditions the OTD modes converge exponentially fast to the eigendirections of the Cauchy--Green tensor associated with the largest eigenvalues, i.e. with the largest finite-time Lyapunov exponents. Thus, we prove that the OTD modes are able to extract over a finite-time interval the modes associated with the most intense finite-time instabilities or the most pronounced transient phenomena. Note, that such convergence does not depend on the dimensionality of the system nor the intrinsic dimensionality of the attractor. 

Apart from the fundamental implications of this result on the formulation of reduced-order algorithms for the prediction and control of transient phenomena there is a direct application on the computation of the maximum finite-time Lyapunov exponents for high- or infinite-dimensional systems. Specifically, using the OTD modes we formulate a reduced-order framework that allows, under mild conditions, the computation of finite-time Lyapunov Exponents for general dynamical systems without having to compute the full Cauchy--Green tensor. We demonstrate the derived algorithm in two systems, the Arnold-Beltrami-Childress (ABC) flow and the Six-dimensional Charney-DeVore model with regime transitions. In both cases we examine the effectiveness of the reduced-order method and study limitations and applicability.

\section{Preliminaries and Notation}
We consider the system of differential equations
\begin{equation}\label{eq:dyn_sys}
\dot{z} = f(z,t), \quad z \in \mathbb{R}^n, \quad t\in I=[t_0,t_0+T],
\end{equation}
where $f:U \times I \rightarrow \mathbb{R}^n$ is a sufficiently smooth vector field. 
Let $F_{t_0}^t$ with $t\in I$ denote the associated flow map,
\begin{align}
F_{t_0}^t : &U \rightarrow U \\ \nonumber
            &z_0 \mapsto z(t;t_0,z_0),
\end{align}
where $z(t;t_0,z_0)$ is a trajectory of system~\eqref{eq:dyn_sys} with the initial condition $z_0$.
The linearized system around the trajectory $z(t)\equiv z(t;t_{0},z_0)$ satisfies the equation of variations,
\begin{equation}\label{eq:lin_dyn_sys}
\dot{v} = L( z(t),t) v, \quad v(t) \in \mathbb{R}^n, \quad t\in I=[t_0,t_0+T],
\end{equation}
where $L(z,t):=\nabla_z f(z,t)$.
The deformation gradient $\nabla F^t_{t_0}$ is the fundamental solution matrix for the linearized dynamics
such that the solutions $v(t)\equiv v(t;t_0,v_0)$ of the equation of variations satisfy
\begin{equation}\label{eq:fund_sol_matrix}
v(t) = \nabla F_{t_0}^t(z_0) v_0, \quad \mbox{with}  \quad t\in I=[t_0,t_0+T].
\end{equation}

\subsection{Finite-time Lyapunov exponents}
To measure the growth of infinitesimal perturbations in the phase space we use 
the \emph{right Cauchy–Green strain tensor}
 \begin{equation*}
C_{t_0}^t = (\nabla F_{t_0}^t)^T \nabla F_{t_0}^t,
\end{equation*}
and the \emph{left Cauchy–Green strain tensor}
\begin{equation}
B_{t_0}^t = \nabla F_{t_0}^t (\nabla F_{t_0}^t)^T,
\end{equation}
where $T$ denote matrix transposition. Let $\xi(t,t_0,z_0)\in\mathbb R^n$
denote the eigenvectors of the right Cauchy--Green strain tensor, so that 
 \begin{equation*}
C_{t_0}^t(z_0) \xi_i(t,t_0,z_0)= \lambda_i(t,t_0,z_0) \xi_i(t,t_0,z_0),\qquad i=1,\cdots, n,
\end{equation*}
where $\lambda_i(t,t_0,z_0)$ are the corresponding eigenvalues.
Similarly, denote the eigenvectors of the left Cauchy--Green strain tensor by $\eta_i(t,t_0,z_0)\in\mathbb R^n$
and their corresponding eigenvalues by $\mu_i(t,t_0,z_0)$, so that
 \begin{equation*}
B_{t_0}^t(z_0) \eta_i(t,t_0,z_0) = \mu_i(t,t_0,z_0) \eta_i(t,t_0,z_0),\quad i=1,\cdots, n.
\end{equation*}
When no confusion may arise, we omit the dependence of the eigenvalues and eigenvectors on $(t,t_0,z_0)$ for notational simplicity.
Since the Cauchy--Green strain tensors are symmetric and positive definite, 
their eigenvalues are real and positive, $\lambda_i, \mu_i \in \mathbb{R}^+$.
Furthermore, the eigenvectors are orthogonal, i.e.,
 \begin{equation*}
\big< \xi_i, \xi_j \big > = \big< \eta_i, \eta_j \big > = \delta_{ij}, \quad i,j=1,\dots,n,
\end{equation*}
where $\langle \cdot,\cdot\rangle$ denotes the Euclidean inner product.

It is straightforward to show that the right and left tensors have the same eigenvalues, $\lambda_i=\mu_i$. 
Specifically, considering the definition of the right Cauchy--Green tensor eigenvalues,
$(\nabla F_{t_0}^t)^T \nabla F_{t_0}^t\xi_i= \lambda_i \xi_i$,
we multiply the equation with $\nabla F_{t_0}^t$ from the left to obtain
\begin{equation}
B_{t_0}^t \left(\nabla F_{t_0}^t\xi_i\right)= \lambda_i \left(\nabla F_{t_0}^t\xi_i\right).
\end{equation}
Thus, the left and right Cauchy--Green strain tensors have the same set of eigenvalues. Finally, using the singular value decomposition of
the deformation gradient, one can show the well-known relation that
 \begin{equation}\label{eq:eigendecom_sol}
\nabla F_{t_0}^t \xi_i=\sqrt{\lambda_i} \eta_i, \quad \quad i=1,\dots, n,
\end{equation}
(see, e.g.,~\cite{karrasch2014attraction}). In the following, we order the Cauchy--Green eigenvalues in a descending order,
\begin{equation}
\lambda_1 > \lambda_2 > \dots > \lambda_n>0.
\end{equation}
The finite-time Lyapunov exponents (FTLEs) of system~\eqref{eq:dyn_sys} corresponding to the trajectory $z(t;t_0,z_0)$
and the time interval $[t_0,t_0+T]$ are defined as
\begin{equation}
\Lambda_i(t_0+T,t_0,z_0)=\frac{1}{T}\log \sqrt{\lambda_i(t_0+T,t_0,z_0)},\quad i=1,2,\cdots,n.
\end{equation}

Note that FTLEs are well-defined for any finite integration time $T$. 
Under certain assumptions, the limit $T\to\infty$ also exists~\cite{ott2008}.

\subsection{Optimally Time-Dependent (OTD) modes}
We  give an overview of the  OTD modes \cite{Babaee} for general dynamical systems. 
The OTD modes represent a reduced-order set of $r$ time-dependent orthonormal  
modes, $U(t)=[u_{1}\left(t\right),u_{2}\left(t\right),...,u_{r}\left(t\right)]$,
that minimize the difference between the action of the infinitesimal propagator $\nabla F^{t+\delta t}_{t}$ on $u_i(t)$ and 
its value at time $t+\delta t$. More specifically, we seek to minimize the functional
\begin{align}\label{eq:functional_0}
\mathcal{F} =\frac{1}{\left(\delta t\right)^2} \sum_{i=1}^{r} \big\|u_i(t+\delta t)- \nabla F_{t}^{t+\delta t}u_i(t)\big\|^2, \quad \quad \delta t \rightarrow 0,
\end{align}
with the constraint that the time-dependent basis satisfies the orthonormality conditions
\begin{align}
\left\langle u_{i}\left(t\right), u_{j}\left(t\right)  \right\rangle=\delta_{ij},\qquad i,j=1,...,r.
\label{eq:orthono}
\end{align} 
Using Taylor series expansions, we have
\begin{align}
u_i(t+\delta t) &= u_i(t)+\delta t\; \dot{u}_i + \mathcal{O}(\delta t^2),\\
\nabla F_{t}^{t+\delta t} & = I+\delta t\; L(z(t),t) + \mathcal{O}(\delta t^2),
\end{align}
where $I$ is the identity matrix.
Replacing the above equations into the functional (\ref{eq:functional_0}) results in
\begin{equation}
\mathcal{F}(\dot u_1,\dot u_2, \dots,\dot u_r) 
 =\sum_{i=1}^{r}\left\Vert \frac{\partial u_{i}\left(t\right)}{\partial
t}-L({z({t)},t)u_i(t)}\right\Vert ^{2}.
\label{eq:functional}
\end{equation}
% ------------------------------ Old explanation -----------------------------------------------
% Our aim  is to evolve a  basis $u_{i}, i=1,...,r$, i.e. a set of
% time-dependent, orthonormal modes, so that the rate of change of these modes, $\dot u_i$, approximates in the best possible way the image of the operator $Lu_i$. In other words, we want the vectors $\dot U$ to be as close as possible, essentially to represent the linearized dynamics, $L({S_{t},t)}U$, while $U$ satisfies the orthonormalization constraints. To achieve this goal we formulate

We emphasize that the minimization of the function (\ref{eq:functional}) is considered only with respect to the time-derivative (rate of change) of the basis, $\dot U(t),$ instead of the basis $U(t)$ itself. This is because we do \textit{not} want to optimize the subspace that the operator is acting on, but rather find an optimal set of vectors, $\dot U(t),$ that best approximates the  linearized dynamics in the subspace $U$. We then solve the resulted equations and compute $U(t)$. We will refer to these modes as the \textit{optimally time-dependent} (OTD) modes, and the space that these modes span as the OTD subspace.
By utilizing the minimization principle and taking into account the orthonormality constraint we obtain the following theorem (proved in \cite{Babaee}).
\begin{theorem} 
A one parameter family of vectors $u_i(t)\in\mathbb R^{n}$, $i=1,2,\cdots,r$, minimizes the
functional (\ref{eq:functional}) and satisfies the orthonormality condition~\eqref{eq:orthono}
if and only if
\begin{equation}
\frac{\partial U}{\partial t} =L{U}-UU^{T}LU,
\label{eq:DO_basis}
\end{equation}
where $U(t)=[u_{1}\left(t\right),u_{2}\left(t\right),...,u_{r}\left(t\right)]\in\mathbb R^{n\times r}$.
\end{theorem}

Equation \eqref{eq:DO_basis} is the evolution equation for the OTD modes.
It was proven in \cite{Babaee} that for the case of a time independent operator, $L$, the subspace
spanned by the columns of $U(t)$ converges asymptotically to the modes 
associated with the most unstable directions of the operator $L$. 

The OTD modes were also used to capture 
the transient non-normal growth of instabilities. In figure~\ref{fig_low}, we recall a system from~\cite{Babaee} 
exhibiting a non-normal growth (cf.~\cite{Babaee} for the details). 
A typical trajectory of the system, shown in the left panel, has an almost periodic behavior where each cycle exhibits three distinct regimes: (A) an exponential  growth in 
the $z_3$ direction, (B) a non-normal growth in the $z_1-z_2$ plane with simultaneous decay in the $z_3$ direction and, (C) exponential decay in the $z_1-z_2$ plane  to 
the origin. This configuration allows for the repeated occurrence of exponential (regime A) and non-normal (regime B) instabilities.

Due to the exponential instability close to the origin the system undergoes chaotic transitions between positive and negative values of $z_3$. We compute the instantaneous growth rate corresponding to the direction of a single OTD mode. The OTD mode initially captures the severe exponential growth and subsequently captures the non-normal growth. On the other hand the real part of the eigenvalues of the full linearized operator can only capture the exponential growth, even in regimes where it is not relevant, while they completely miss the non-normal growth. The maximum largest singular value $\sigma$ of the matrix $L$ exhibits a similar behavior.

\begin{figure}
\centering
\includegraphics[width=\textwidth]{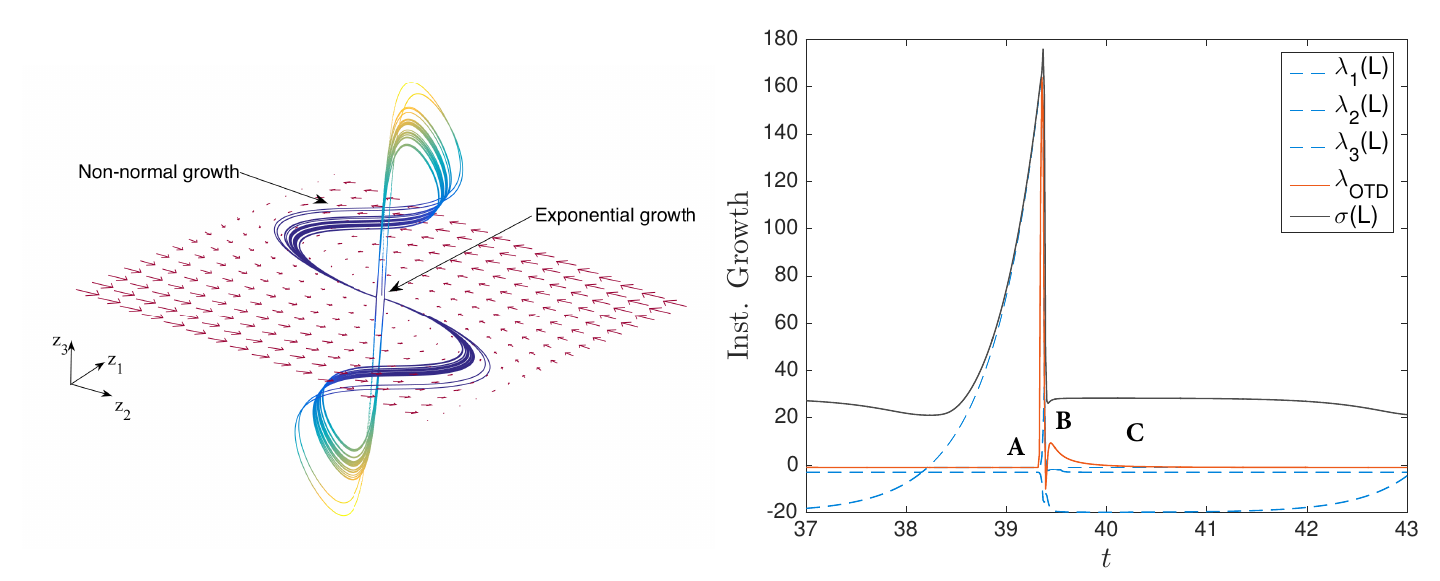}
\caption{An example of the OTD modes on the analysis of transient instabilities. Left: A trajectory of the  dynamical system colored according to the state variable $z_3$. The non-normal vector field for $z_3=0$ is also shown. Right: The three eigenvalues of the linearized operator are plotted with blue dashed curves while the growth rate of the single OTD mode is shown with orange color. The  maximum eigenvalue of the singular value decomposition of the linearization along the trajectory is also shown.} 
\label{fig_low}
\end{figure}

The OTD modes share some fundamental characteristics of other stability measures. In the next section,
we show the relation between the OTD modes and the finite-time Lyapunov vectors.
In Appendix A, we also discuss the connection between them and the Dynamically Orthogonal (DO) modes
introduced in~\cite{SapsisLermusiaux09}. Their connection with
covariant Lyapunov vectors~\cite{ginelli07} remains to be explored. 

% --------------------------------------------------------------------------------------------------------------
% --------------------------------------------------------------------------------------------------------------
% --------------------------------------------------------------------------------------------------------------
\section{Alignment of the OTD modes with the dominant Cauchy--Green strain eigenvectors}
Here we prove that under  proper assumptions, the OTD modes not only capture the unstable directions of steady linear operators, but also the transient  instabilities associated with time-dependent operators. More specifically, we prove that the OTD subspace converges exponentially fast to the eigenspace associated with the
dominant eigenvalues of the Cauchy--Green strain tensor. The latter describes the finite-time growth of infinitesimal perturbations in the phase space.   
For what follows we will need to measure the distance between two
subspaces of the same dimensionality. To this end we need the following definitions. 
\begin{definition}
Two $r$-dimensional linear subspaces of $\mathbb{R}^{n}$ 
spanned by the columns of $U  \in \mathbb{R}^{n\times r}$ and  $V  \in \mathbb{R}^{n\times r}$ 
are equivalent if there is an invertible matrix $R \in \mathbb{R}^{r\times r}$ 
such that $U=VR$.
\end{definition}

The next definition provides a quantity for measuring the `angle' between two subspaces.

\begin{definition}\label{def:subspace_dist}
For two $r$-dimensional linear subspaces of $\mathbb{R}^{n}$ spanned by the
columns of 
$U=[u_{1},u_{2},...,u_{r}]  \in \mathbb{R}^{n\times r}$ and 
$V=[v_{1},v_{2},...,v_{r}]  \in \mathbb{R}^{n\times r}$, we define the distance
function
\begin{equation}
\gamma_{U,V}=\frac{\Vert U^T V \Vert}{r^{1/2}},
\end{equation}
where $\Vert \cdot \Vert$ denotes the Frobenius norm.
\end{definition}

Note that the entries of $U^T V$ are the inner products between $u_i$'s and
$v_j$'s, i.e., $(U^TV)_{ij}=\inner{u_i}{v_j}$.
The following lemma
shows that the equivalence of two subspaces can be deduced from the distance
$\gamma_{U,V}$.
\begin{lemma}\label{lemma1}
Given a subspace defined by the orthonormal columns of $U\in \mathbb{R}^{n\times r}$, and a subspace defined by the unit-length (but not necessarily orthonormal) columns of $V\in \mathbb{R}^{n\times r}$,  
then we always have $\gamma_{U,V} \le 1$. Moreover, the two subspaces spanned by the columns of $U$ and $V$ 
are equivalent if and only if $\gamma_{U,V}=1$.
 \end{lemma} 
\begin{proof}
By the orthogonal projection theorem, there are $h_i\in\mathbb R^r$ and $b_i\in\mathbb R^n$ such that
$v_i=Uh_i+b_i$ and $U^Tb_i=0$ for $i=1,2,\cdots, r$. Note that since $v_i$ are unit length, 
we have 
$$1=\|v_i\|^2=\|h_i\|^2+\|b_i\|^2,\quad i=1,2,\cdots,r.$$
Defining $H=[h_1|\cdots|h_r]$ and $B=[b_1|\cdots|b_r]$, we have $V=UH+B$.
This implies $U^TV = H$ and hence 
$$\|U^TV\|^2 = \|H\|^2=\sum_{i=1}^{r}\|h_i\|^2=r-\sum_{i=1}^{r}\|b_i\|^2=r-\|B\|^2.$$
Therefore, $\gamma_{U,V}^2=1-\|B\|^2/r$. 
Note that the subspaces $\mathrm{col}(U)$  and $\mathrm{col}(V)$ are equal if and only if $\|B\|=0$.
Therefore, $\mathrm{col}(U)=\mathrm{col}(V)$ if and only if $\gamma_{U,V}=1$. 
We also note that if the two subspaces are not equivalent $\gamma_{U,V}<1.$ 
\end{proof}

To prove the main theorem it is more convenient to work with the subspace
spanned by the linearized flow~\eqref{eq:lin_dyn_sys}, instead of the OTD subspace. 
The following result states that the two subspaces are equivalent.

\begin{theorem}
\label{Thm:Trans}
Let $V(t)\in \mathbb{R}^{n\times r}$ solve the equation of variations~\eqref{eq:lin_dyn_sys} 
and $U(t)\in \mathbb{R}^{n\times r}$ with $U^TU=I$ solve the OTD equation~\eqref{eq:DO_basis}.
Assume that the two subspaces spanned by the columns of $U$ and $V$ are initially equivalent, i.e. there
is an invertible matrix $T_0\in\mathbb{R}^{r\times r}$ such that $V_0 = U_0T_0$. 
Then there exists a one-parameter family of linear transformations
$T(t)$ such that $V(t) = U(t)T(t)$ for all $t$ and $T(t)$ satisfies the
differential equation
\begin{equation}\label{eq:Tdot}
\dot{T}=L_r(t)T,\quad T(t_0)=T_0,
\end{equation}
where $L_r=U^T L U$ is the orthogonal projection of $L$ to the linear subspace spanned by the columns of $U$.
\end{theorem}
\begin{proof}
 See Theorem 2.4 in \cite{Babaee}.
\end{proof}

Based on this equivalence, it is  sufficient to study the behavior of the subspace evolved under the
time-dependent linearized dynamics~\eqref{eq:lin_dyn_sys} in order to understand the properties of the OTD subspace.
Specifically, we use this lemma  to prove the main result of this paper, namely, that under a spectral gap condition, 
the OTD subspace will converge to the dominant directions of the Cauchy--Green strain tensor.

\begin{theorem}\label{Thm:Alignment}
Let $\xi_i(t,t_0,z_0)$ and $\eta_i(t,t_0,z_0)$ $(i=1,2, \dots, n)$ denote the eigenvectors 
of the right  and left Cauchy--Green strain tensors, 
respectively, with corresponding eigenvalues  $\lambda_1 (t,t_0,z_0)\geq\ \lambda_2 (t,t_0,z_0)\geq \dots \geq\lambda_n(t,t_0,z_0)$. 
Moreover, let $U(t) \in \mathbb{R}^{n \times r}$ solve the OTD equation~\eqref{eq:DO_basis}
along the trajectory $z(t;t_0,z_0)$. Assume that

i) The subspace $U(t)$ is initiated so that each basis element $u_i(t_0)$ 
has a non-zero projection on at least one of the eigenvectors   
$\xi_i(t,t_0,z_0)$, $i=1,2,\cdots, r$ for all times $t$.

ii) The trajectory $z(t;t_0,z_0)$ is hyperbolic in the sense that there exist constants 
$a_1\geq a_2\geq\cdots\geq a_r>a_{r+1}\geq \cdots\geq a_n$
and $K_i>0$ $(i=1,2,\cdots,n)$ such that
\begin{equation}
\lim_{t\to\infty}\frac{\lambda_i(t,t_0,z_0)}{e^{a_i t}}=K_i.
\end{equation}

%ii) There is no eigenvalue crossing between $\lambda_r (t,t_0,z_0)$ and $\lambda_{r+1} (t,t_0,z_0)$, 
%i.e. $\lambda_r (t,t_0,z_0)>\lambda_{r+1}(t,t_0,z_0)$ for all $t>t_0$.

Then the subspace $U(t)$ aligns with the $r$ most dominant left Cauchy--Green strain eigenvectors, 
$\eta_i(t,t_0,z_0)$, $i=1,...,r$, exponentially fast as $t\to\infty$. 
\end{theorem}
 
\begin{proof}
Based on Lemma \ref{Thm:Trans} the subspace spanned by the OTD modes is equivalent to the subspace 
that we obtain if we evolve $V(t)=[v_1(t),...,v_r(t)]$ using the linearized dynamics \eqref{eq:lin_dyn_sys}. 
We represent the initial condition for the equation of variations~\eqref{eq:lin_dyn_sys} as 
$V(t_0)=[v_{0_1},v_{0_2},...,v_{0_r}]$ and express it in terms of the orthonormal strain eigenbasis $\{\xi_i(t,t_0,z_0) \}$ as
\begin{equation}\label{eq:Alignment_aux_1}
v_{0_j} = \sum_{i=1}^n \inner {v_{0_j}}{\xi_i(t,t_0,z_0)} \xi_i(t,t_0,z_0), \quad \ j=1,...,r.
\end{equation}
Then,  equations (\ref{eq:fund_sol_matrix}) and (\ref{eq:eigendecom_sol})  applied to equation (\ref{eq:Alignment_aux_1}) yield
\begin{equation*}
v_{j}(t) = \sum_{i=1}^n\sqrt{\lambda_i(t,t_0,z_0)} \inner {v_{0_j}}{\xi_i(t,t_0,z_0)} \eta_i(t,t_0,z_0), \quad \ j=1,...,r.
\end{equation*}
We then have
\begin{equation*}
\inner {v_{j}(t)}{\eta_k(t,t_0,z_0)} = \sqrt{\lambda_k(t,t_0,z_0)} \inner {v_{0_j}}{\xi_k(t,t_0,z_0)}
, \quad \ j,k=1,...,r.
\end{equation*}
Moreover,
\begin{equation*}
\inner {v_{j}(t)}{v_{j}(t)} = \sum_{i=1}^n{\lambda_i(t,t_0,z_0)}
\inner {v_{0_j}}{\xi_i(t,t_0,z_0)}^2
, \quad \ j=1,...,r.
\end{equation*}
Based on this last equation, the cosine of the angle $\alpha_{v_j,\eta_k}(t,t_0,z_0)$ between
the vector $v_j(t)$ and the  eigenvector $\eta_k(t,t_0,z_0)$ of $B_{t_0}^t(z_0)$  obeys the
relation
\begin{equation}\label{eq:Alignment_aux_2}
\cos{\alpha_{v_j,\eta_k}}= \frac{\inner{v_j(t)}{\eta_k}}{\|v_j(t)\| }
              =\frac{\sqrt{\lambda_k}\inner{v_{0_j}}{\xi_k}}{\sqrt{ \sum_{i=1}^n \lambda_i\inner{v_{0_j}}{\xi_i}^2 }}.
\end{equation}
Identity \eqref{eq:Alignment_aux_2} and definition~\ref{def:subspace_dist} yield
\begin{equation}\label{eq:Alignment_aux_3}
r\gamma_{v,\eta}^2
=\sum_{j=1}^r\sum_{k=1}^r \frac{\lambda_k\inner{v_{0_j}}{\xi_k}^2}{ \sum_{i=1}^n
\lambda_i\inner{v_{0_j}}{\xi_i}^2 }
=\sum_{j=1}^r \frac{\sum_{k=1}^r \lambda_k\inner{v_{0_j}}{\xi_k}^2}{\sum_{i=1}^n
\lambda_i\inner{v_{0_j}}{\xi_i}^2 }
=\sum_{j=1}^r \frac{A_j}{A_j+B_j}.
\end{equation}
where
\begin{align*}
A_j&=\sum_{k=1}^r \lambda_k\inner{v_{0_j}}{\xi_k}^2\geq \lambda_r \sum_{k=1}^r \inner{v_{0_j}}{\xi_k}^2, \\
B_j&=\sum_{i=1}^n\lambda_i\inner{v_{0_j}}{\xi_i}^2\leq \lambda_{r+1} \sum_{i=r+1}^n.
\inner{v_{0_j}}{\xi_i}^2
\end{align*}
Note that assumption (ii) implies $\lim_{t\to\infty}\lambda_{r+1}(t,t_0,z_0)/\lambda_{r}(t,t_0,z_0)=0$ since
\begin{align*}
\lim_{t\to\infty}\frac{\lambda_{r+1}(t,t_0,z_0)}{\lambda_r(t,t_0,z_0)}  &= 
\lim_{t\to\infty}\frac{\lambda_{r+1}(t,t_0,z_0)/e^{a_{r+1}t}}{\lambda_r(t,t_0,z_0)/e^{a_{r}t}}\frac{e^{a_{r+1}t}}{e^{a_{r}t}}\nonumber\\
& = \frac{K_{r+1}}{K_r}\lim_{t\to\infty}e^{(a_{r+1}-a_r)t}.
\end{align*}
As a consequence, we obtain from the last inequality the estimate
\begin{equation}\label{eq:Alignment_aux_4}
r\gamma_{v,\eta}^2
=\sum_{j=1}^r \frac{1}{1+\frac{B_j}{A_j}} \geq 
\sum_{j=1}^r \frac{1}{1+\frac{\lambda_{r+1}}{\lambda_{r}}\frac{ \sum_{i=r+1}^n
\inner{v_{0_j}}{\xi_i}^2}{\sum_{k=1}^r
\inner{v_{0_j}}{\xi_k}^2}}.
%&\geq
%\sum_{j=1}^r \frac{1}{1+Ke^{-2bt}\frac{ \sum_{i=r+1}^n
%\inner{v_{0_j}}{\xi_i}^2}{\sum_{k=1}^r
%\inner{v_{0_j}}{\xi_k}^2}}.
\end{equation}
Note that $\sum_{k=1}^r
\inner{v_{0_j}}{\xi_k}^2>0$ from the first assumption. Therefore, for the asymptotic limit $t\to\infty$, we have
\begin{displaymath}
\gamma_{v,\eta}^2 \geq 1.
\end{displaymath}
Note that the way we have defined the two subspaces satisfy the assumptions of Lemma 3.1 and therefore we always have $\gamma_{v,\eta}^2 \leq 1$. 
Thus, $\gamma_{v,\eta}^2=1$ which implies that the two subspaces are equivalent. This completes the proof.
\end{proof}

We point out that the hyperbolicity condition (ii) is not a necessary condition. In fact, as long as the ratio $\lambda_{r+1}(t)/\lambda_r(t)$ tends to 
zero asymptotically, the alignment takes place (see equation~\eqref{eq:Alignment_aux_4}).\textcolor{red}{}

\section{Eigenvalue crossing and OTD modes}
Theorem~\ref{Thm:Alignment} shows that under appropriate assumptions the OTD modes converge exponentially fast to the eigenspace of the left Cauchy--Green
strain tensor associated with the largest Lyapunov exponents. 
This convergence, however, may be interrupted when the smallest eigenvalue spanned by the OTD modes crosses with the one that is not being spanned. 
In figure~\ref{fig:schem_crossing}, we illustrate this situation. The green curve denotes the trajectory of the system, while the ellipses indicate the principal axes of the 
left Cauchy--Green strain tensor. Assuming that we are resolving only one OTD mode, after sufficient time this must have  converged to the correct principal direction (red 
arrow). As we go through the critical time instant where the ellipsoid becomes a circle due to eigenvalue crossing, we observe that the principal directions are 
instantaneously ill-defined. Immediately after the eigenvalue crossing, the dominant principal direction undergoes a $90$ degrees internal rotation compared
to immediately before the crossing. The OTD modes, on the other hand, evolve smoothly and are unable to capture such an instantaneously unbounded transition. 
\begin{figure}
\centering
\includegraphics[width=.68\textwidth]{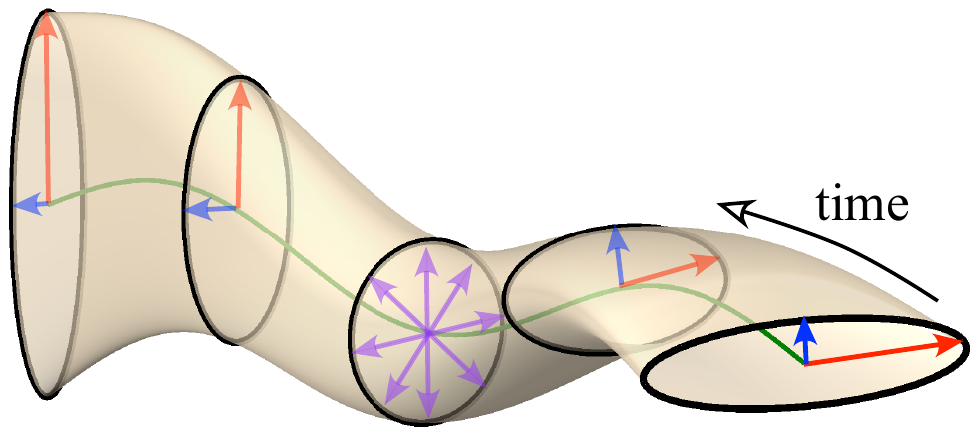}
\caption{Evolution of the principal eigenvector (red) of the Cauchy--Green tensor. When eigenvalue crossing occurs the rate of change becomes unbounded. This is not a problem if all the eigenvectors are resolved but it is important if we evolve only the dominant eigenvectors as it is the case with the OTD modes.}
\label{fig:schem_crossing}
\end{figure}

In the context of shearless Lagrangian coherent structures, these eigenvalue crossings are referred to as the \emph{Cauchy--Green singularities} and play an
important role in the computation of jet cores~\cite{shearless}. In a more general application, the significance of these singularities is 
pointed out by Lancaster~\cite{lancaster} who derived the rate of change of the eigenvectors of an arbitrary symmetric matrix, proving the
following theorem.
\begin{theorem}\label{thm:dotR}
The rate of change of the eigenvectors $R(t) \in \mathbb{R}^{n\times n}$ of a symmetric matrix $G(t)\in \mathbb{R}^{n\times n}$ is described  by the 
equation\begin{equation}
\dot{R} = R K,
\end{equation}
where $K(t)$ is a skew-symmetric matrix given by
\begin{equation}
 K_{ij} (t)= \frac{\widetilde{G}_{ij}
(t)+ \widetilde{G}_{ji}(t)}{2(\lambda_j (t)- \lambda_i(t))}
, \quad i \neq j
\end{equation}
with $\widetilde{G} (t)= R^T(t) \dot G(t)R(t)$  and $\lambda_i(t) $ being the eigenvalues of $G(t)$.
\end{theorem}
\begin{proof}
See Appendix B or \cite{lancaster}.
\end{proof}

As $\lambda_i$ approaches $\lambda_j$, $\dot{R}$ becomes unbounded, consistent with the description given previously.  This is not a problem if both of the eigen-directions corresponding to $\lambda_i$ and $\lambda_j$ are being spanned/resolved by the OTD subspace. The corresponding OTD subspace will just evolve smoothly in this case. However, if the eigenvalue crossing occur with an eigendirection that is not being spanned by the OTD subspace, the latter will have to instantaneously `steer' towards a direction that is orthogonal to it. 

This is not possible since it can be easily seen that for any bounded operator $L,$ the rate of change of the corresponding OTD modes (eq. (\ref{eq:DO_basis})) is bounded:
\begin{equation}\label{eqn:OTD_bounded}
\dot{U} = QLU\Rightarrow \|\dot{U} \| \leq \|Q\| \| L\| \|U\|,
\end{equation}
where $Q$ is a projection operator (therefore bounded) and $U$ is also bounded. If the eigenvalue crossing involves modes already contained in $U$ then there is no need for infinite rate of change of the modes, because existing modes just exchange roles (or indices). However, if there is eigenvalue crossing between a mode that is contained in $U$  and another one that is not contained then it is essential to have an evolution of the direction in $U$ to an orthogonal direction (not contained in $U$) and this cannot happen instantaneously. 

Therefore, we see that the convergence of the OTD modes can be compromised if there are  eigenvalue crossings involving directions not spanned by the OTD subspace. Note however that such eigenvalue crossing will only influence the convergence of the OTD mode associated with the smallest finite-time Lyapunov exponent captured within the OTD subspace. Assuming that there are no subsequent eigenvalue crossings with larger eigenvalues, one expects that the OTD subspace would capture the directions associated with the largest finite-time Lyapunov exponents.

Thus, for systems where eigenvalue crossing occurs, a  larger OTD subspace  guarantees that the principal axes associated with the largest Lyapunov exponents will  be captured by the OTD modes. Note that \textit{such convergence does not depend on the physical dimension of the system or the intrinsic dimension of the underlying attractor}. This is particularly important for systems where instabilities of isolated modes emerge out of high-dimensional stochastic backgrounds, such as intermittent dissipation bursts in turbulence \cite{Farazmand2016} or focusing instabilities occurring in nonlinear stochastic waves \cite{cousinsSapsis2015_JFM, Farazmand2016}.

As we demonstrate in the next section, for many practical applications it is sufficient to compute just the maximum finite-time
Lyapunov exponent so what is important is to have the most dominant direction of the Cauchy--Green tensor captured by the OTD subspace. The benefits of such approach  are, of course, more important as the dimension $n$ of the underlying system increases.     
% --------------------------------------------------------------------------------------------------------------------------------------------------------
% --------------------------------------------------------------------------------------------------------------------------------------------------------
% --------------------------------------------------------------------------------------------------------------------------------------------------------  
\section{Reduced-order computation of finite-time Lyapunov exponents}
As shown above, the OTD modes are able to extract the transiently most unstable directions independently of the attractor dimensionality.
Consequently, the OTD basis can be used for the reduced-order approximation of finite-time Lyapunov exponents (FLTE) of high-dimensional systems. 
Specifically, we apply the following approximation scheme for the computation of a FTLE corresponding to a time interval of length $T$.
\begin{enumerate}
\item Advect  the trajectory $z(t;t_0,z_0)$ for an interval $t \in[t_0,t_0+T]$ where $z_0$ is the initial point.

\item Compute the $r$-dimensional OTD subspace $U(t)$ corresponding to this trajectory.

\item Compute  the low-dimensional fundamental solution matrix $\Phi_{t_0}^t \in \mathbb{R}^{r\times r}$ using the reduced linear dynamical system
\begin{equation}
\frac{\id}{\id t}\Phi_{t_0}^t = L_{r}( z(t),t)\Phi_{t_0}^t,\quad \Phi_{t_0}^{t_0}=I ,\quad t\in[t_0,t_0+T],
\end{equation}
where $L_{r}(z,t)=U(t)^{T}L(z,t)U(t)$ is the projection of the full linearized operator $L$ onto the OTD subspace, which contains the most unstable directions.
\item
Compute the reduced-order finite-time Lyapunov exponents, 
\begin{equation}
\Gamma_i(t_0+T,t_0,z_0)=\frac{1}{T}\log\sqrt{\gamma_i(t_0+T,t_0,z_0)},\quad i=1,\cdots,r,
\label{reduced_ftle}
\end{equation}
where $\gamma_i$ denotes the 
eigenvalues of the reduced-order right Cauchy--Green strain tensor $\left(\Phi_{t_0}^{t_0+T}\right)^\top \Phi_{t_0}^{t_0+T}$,
ordered such that $\gamma_1\geq \gamma_2\geq\cdots\geq \gamma_n$.
\end{enumerate}
\begin{figure}
\centering
\includegraphics[width=.5\textwidth]{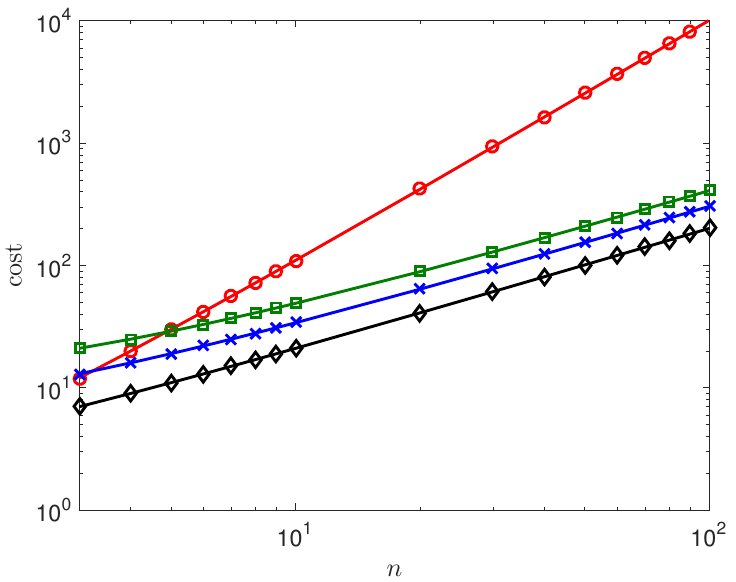}
\caption{The computational cost of the FTLE using the full system~\eqref{eq:full_sys} (red circles),
and the reduced system~\eqref{eq:reduced_sys} with $r=1$ (black diamonds), $r=2$ (blue crosses) and $r=3$ (green squares).
The integer $n$ is the system dimension. The computational cost is measured as the number of scalar ODEs that need to be solved for each case.}
\label{fig:CompCost}
\end{figure}

The described algorithm allows for the computation of FTLE by solving for each initial point, $z_0$, one $n$-dimensional equation, as well as additional $r$ equations which are $n$-dimensional  for the OTD subspace. This cost has to be summed together with the solution of reduced-order linear system. To this end we have the total cost for the reduced-order and full algorithms:
\begin{itemize}
\item
\textbf{Full-order calculation cost:} $(n+1)\times n$ consisting of $n+1$ equations of dimension $n$ as follows:
\begin{subequations}
\begin{equation}
\dot z = f(z,t)\quad \mbox{($n$ equations)},
\end{equation}
\begin{equation}
\frac{\id}{\id t}\nabla F_{t_0}^t(z_0) = L( z(t),t)\nabla F_{t_0}^t(z_0)\quad \mbox{($n\times n$ equations)}.
\end{equation}
\label{eq:full_sys}
\end{subequations}
\item
\textbf{Reduced-order calculation cost:} $n\times (r+1)+r^2$ consisting of $r+1$ equations of dimension $n$ and $r$ equations of dimension $r$
as follows:
\begin{subequations}
\begin{equation}
\dot z = f(z,t)\quad \mbox{($n$ equations)},
\end{equation}
\begin{equation}
\dot U =L(z(t),t)U-U\left(U^{T}L(z(t),t)U\right)\quad \mbox{($n\times r$ equations)},
\label{eq:otd}
\end{equation}
\begin{equation}
\frac{\id}{\id t}\Phi_{t_0}^t = L_{r}( z(t),t)\Phi_{t_0}^t\quad \mbox{($r\times r$ equations)}.
\label{eq:EqVari_reduced}
\end{equation}
\label{eq:reduced_sys}
\end{subequations}
\end{itemize}

We observe that the cost of computing FTLE using the full system is $\mathcal O(n^2)$. The computational cost of the reduced-order FTLE, however, is
$\mathcal O(n)$ for fixed $r$. In principle, the dimension $r$ of the reduced system can be as low as $1$. But, as we will see in what follows, a one-dimensional OTD subspace may lead to \emph{false troughs} in the FTLE field. A higher dimensional OTD subspace returns more accurate estimation of the FTLE field and
often circumvents the false troughs. 

Figure~\ref{fig:CompCost} shows the computational cost of the full FTLE calculation versus the 
reduced-order FTLE calculation. Note that for low-dimensional systems (small $n$) the reduction in the computational
cost is not significant. However, as the system dimension increases ($n\gg r$), the computational cost of the reduced-order
FTLE can be orders of magnitude lower than the full FTLE computations.

% --------------------------------------------------------------------------------------------------------------------------------------------------------
% --------------------------------------------------------------------------------------------------------------------------------------------------------
% --------------------------------------------------------------------------------------------------------------------------------------------------------
\subsection{The ABC Flow}
As the first example we consider the ABC (Arnold–Beltrami–Childress) flow with three dimensional velocity field
\begin{align}
\dot{z}_1 &= A \sin(z_3) + C \cos(z_2),\\ \nonumber
\dot{z}_2 &= B \sin(z_1) + A \cos(z_3),\\ \nonumber
\dot{z}_3 &= C \sin(z_2) + B \cos(z_1),
\end{align}
which is an exact solution of the Euler equation for the ideal incompressible fluids~\cite{topolHydro_arnold}. The ABC flow has 
served as a prototype example for testing numerical methods for computing Lagrangian coherent 
structures~\cite{haller01,stretchlines,pra}. Here, we set $A = \sqrt{3}$, $B = \sqrt{2}$ and $C = 1$ which allows for the existence of 
chaotic Lagrangian trajectories~\cite{haller01}. 

We compute the FTLE on a $251 \times 251$ grid of initial conditions on each face of the cube $[0, 2\pi]^3$ with an integration time 
of length $T=8$. 
We use finite differences to approximate the deformation gradient $\nabla F_{t_0}^{t_0+T}$. To increase the finite difference accuracy, 
we use six auxiliary grids $(z_{0,1}\pm h, z_{0,2}\pm h, z_{0,3}\pm h)$ around each grid point $z_0= (z_{0,1}, z_{0,2}, z_{0,3})$. The 
parameter $h$ controls the accuracy of the finite differences and is set to $h=10^{-8}$ in the following. 
We refer to~\cite{lekien2010,computeVariLCS} for further details on the approximation
of the deformation gradient. Once the deformation gradient is approximated, computing the Cauchy--Green strain tensor and 
consequently the FTLE is straightforward. Note that the deformation gradient can also be obtained by numerically integrating 
the equation of variations~\eqref{eq:lin_dyn_sys}.

To compute the reduced-order FTLE, we numerically integrate system~\eqref{eq:reduced_sys} 
from the initial grid points $z_0$ with an integration time of length $T$. 
The initial condition $U(0)$ for the OTD equation~\eqref{eq:otd} is the $r$ most dominant eigenvectors of the symmetric linear 
operator at $t=0$, i.e. $L^s = (L(z_0,0)+L(z_0,0)^T)/2$. This choice of the OTD initial condition is made since these eigenvectors
are the instantaneously most unstable directions~\cite{Haller2010}.

Figure \ref{Fig:ABC_FTLE} shows the reduced-order FTLE fields using OTD reduction of sizes $r=1$ (left column) and $r=2$ 
(middle column). The true FTLE field (right column) is also shown for comparison. 
Each row of the figure shows a different cross section of the field.
For $r=2$, the reduced-order FTLE is close to the true FTLE field.
This is also true qualitatively for $r=1$. However, for $r=1$, the reduced FTLE exhibits
additional troughs that do not exist in the true FTLE field. 

We demonstrate in figure \ref{Fig:ABC_FTLE_His} that these false troughs occur when there is a 
crossing (or near crossing) between the first and second most dominant eigenvalues of the right Cauchy--Green tensor, i.e.,
$\lambda_1$ and $\lambda_2$. 
Recall that this eigenvalue crossing violates condition (ii) of Theorem~\ref{Thm:Alignment} and hence the discrepancy between the
reduced-order FTLE and the true FTLE is expected.
The false troughs are eliminated as we increase the dimension of the OTD reduction from $r=1$ to $r=2$. 
This is due to the fact that there is not eigenvalue crossing between $\lambda_2$ and $\lambda_3$. 
\begin{figure}
\includegraphics[width=\textwidth]{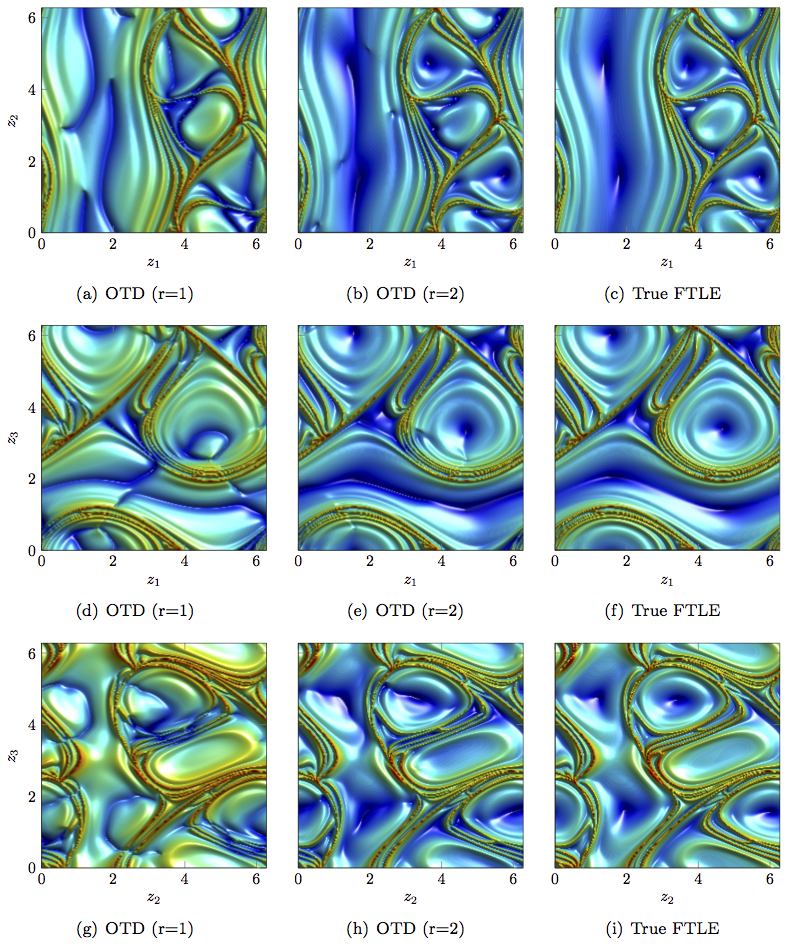}
\caption{ABC flow: comparison of the true FTLE (right column) with OTD reduction of size $r=1$ (left column) and $r=2$ (middle column). }
\label{Fig:ABC_FTLE}
\end{figure}

The above observation is demonstrated in figure~\ref{Fig:ABC_FTLE_His} where the reduced-order FTLEs at two select points 
are compared against the true FTLE as the integration time $T$ varies. In figure \ref{Fig:ABC_his1}, the initial point $z_0=(2.0,0.6,0.0)$ is chosen to be at one of the false troughs. An eigenvalue crossing occurs at approximately $T=0.2$. For advection times before $T=0.2$, both one- and two-dimensional reductions follow the largest FTLE very closely. However, near the eigenvalue crossing the most dominant eigenvector of the Cauchy--Green strain tensor undergoes rotation with an unbounded rate. As demonstrated in inequality \ref{eqn:OTD_bounded}, the one-dimensional OTD reduction cannot follow such rapid rotation.  This is confirmed in figure \ref{Fig:ABC_his1} where it can be seen that the FTLE  obtained from the one-dimensional OTD reduction, after $T=0.2$, follows the second most dominant FTLE instead of undergoing a rapid change that is required to follow the most dominant FTLE. Near $T=1$, the FTLE of the one-dimensional reduction departs from non-dominant FTLEs and starts increasing. This behavior is to be expected since in the remaining advection time no eignevalue crossing occurs and as a result the one-dimensional reduction continues to recover to converge to the most dominant eigenvector -- however slowly.   

 The two-dimensional OTD reduction, however, recovers much faster compared to the one-dimensional OTD reduction after $T=0.2$. We observe that  the two eigenvalues of OTD ($r=2$)  approach  each other nearly at $T=0.5$ causing in a relatively rapid repulsion of these two eigen-directions. This repulsion results in a rapid rotation of  the dominant OTD vector towards the dominant eigen-direction of the Cauchy--Green tensor. This mechanism of rapid correction of the OTD vectors is entirely absent in the one-dimensional OTD reduction.

Figure \ref{Fig:ABC_his2} shows the evolution of the same quantities along a different ABC trajectory.
Here, eigenvalue crossings are absent. As a result both one-dimensional and two-dimensional OTD reductions closely follow the most dominant FTLE and in the case of $r=2$ both most and second most dominant FTLEs. This  demonstrates that, in the absence of eigenvalue crossing, one-dimensional OTD reduction accurately approximates the most dominant  FTLE.
\begin{figure}
\hspace{-1cm}
\subfigure[$z_0=(2.0,0.6,0.0)$]{
\centering
\includegraphics[width=.48\textwidth]{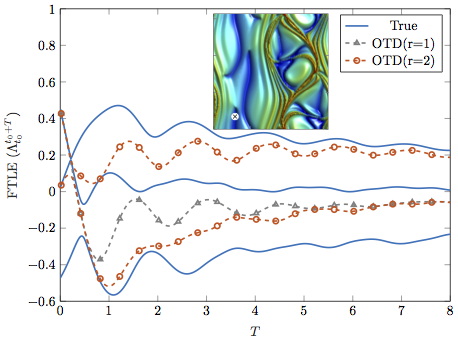}
\label{Fig:ABC_his1}
}
\subfigure[$z_0=(4.0,0.6,0.0)$]{
\centering
\includegraphics[width=.48\textwidth]{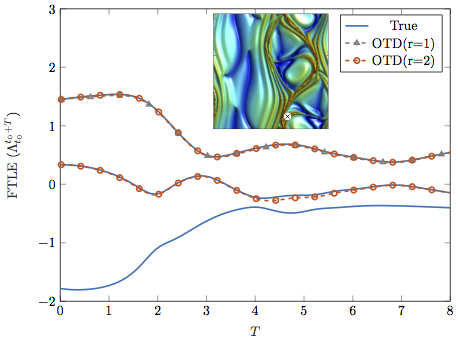}
\label{Fig:ABC_his2}
}
\caption{ABC flow: FTLE versus advection time calculated with full dynamics (solid blue), one-dimensional OTD reduction (dashed gray line with triangle symbols) and two-dimensional OTD reduction (red line with circle symbols) with initial points of: (a)  $z_0=(2.0,0.6,0.0)$, and (b)  $z_0=(4.0,0.6,0.0)$. The initial points are shown with the a cross symbol in the $z_1 - z_2$ plane.}
\label{Fig:ABC_FTLE_His}
\end{figure}
% --------------------------------------------------------------------------------------------------------------------------------------------------------
% --------------------------------------------------------------------------------------------------------------------------------------------------------
% --------------------------------------------------------------------------------------------------------------------------------------------------------
\subsection{Charney–DeVore model with regime transitions}
In this section, we  apply the OTD reduction to the six-dimensional truncation of the equations for barotropic flow in a  plane channel with orography that are known as the Charney–DeVore (CDV) model. The truncated system is given by
\begin{align}
\dot{z}_1 &= \gamma_1^*z_3 - C(z_1 - z_1^*),\\ \nonumber
\dot{z}_2 &= -(\alpha_1 z_1 - \beta_1)z_3 - C z_2 - \delta_1z_4z_6,\\ \nonumber
\dot{z}_3 &=  (\alpha_1 z_1 - \beta_1)z_2 - \gamma_1 z_1 - C z_3 + \delta_1 z_4 z_5,\\ \nonumber
\dot{z}_4 &=  \gamma_2^*z_6 - C(c_4-z_4^*) + \epsilon( z_2 z_6 - z_3z_5),\\ \nonumber
\dot{z}_5 &= -(\alpha_2 z_1 - \beta_2)z_6 - C z_5 - \delta_2 z_4 z_3,\\ \nonumber
\dot{z}_6 &=  (\alpha_2 z_1 - \beta_2)z_2 - \gamma_2 z_4 - C z_6 + \delta_2 z_4 z_2.
\end{align}
The model coefficients are given by
\begin{align}
\alpha_m &= \frac{8\sqrt{2}}{\pi} \frac{m^2}{4m^2-1}\frac{b^2+m^2-1}{b^2+m^2}, \quad \quad \beta_m = \frac{\beta b^2}{b^2+m^2},\\ \nonumber
\delta_m &= \frac{64\sqrt{2}}{15\pi} \frac{b^2-m^2+1}{b^2+m^2}, \quad \quad  \quad\gamma^*_m = \gamma \frac{4m}{4m^2-1}\frac{\sqrt{2}b}{\pi},\\ \nonumber
\epsilon &= \frac{16\sqrt{2}}{5\pi}, \quad \quad \quad \quad \quad \quad \gamma_m = \gamma \frac{4m^3}{4m^2-1}\frac{\sqrt{2}b}{\pi(b^2+m^2)}.
\end{align}
Following~\cite{Crommelin:2004ab}, we set $(z_1^*, z_4^*,C,\beta,\gamma,b) = (0.95, -0.76095, 0.1, 1.25, 0.2, 0.5)$.
For these parameters the CDV model generates regime transitions due  to the interaction of barotropic and topographic instabilities. It was shown in 
\cite{Crommelin:2004aa} that a Proper Orthogonal Decomposition (POD) reduction of this model to three leading POD modes resolves 97\% of cumulative variance but
cannot capture the chaotic regime transitions present in the six-dimensional model. These highly transient instabilities render this model an appropriate test case for 
evaluating the performance of the OTD reduction in computing the FTLE. Moreover, using this model we asses the performance of the OTD reduction for a higher 
dimensional problem. 
\begin{figure}
\includegraphics[width=\textwidth]{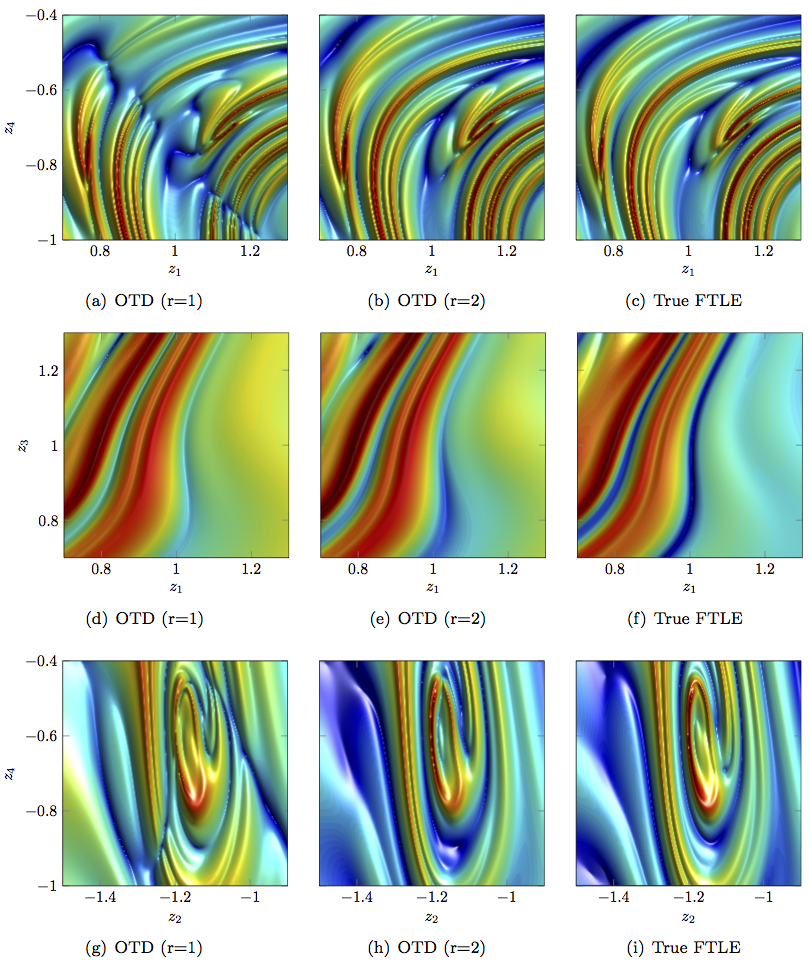}
\label{Fig:CDV_FTLE_Field}
\caption{Six-dimensional Charney–DeVore model: comparison of the true FTLE (rightmost column) with OTD reduction of size $r=1$ (leftmost column) and $r=2$ middle column. Each row shows results of a section of the phase space.}
\end{figure}

We compute the true FTLE using the finite difference method analogous to the ABC flow problem. Similarly, the initial condition $U(0)$ for the
OTD equation~\eqref{eq:otd} are the $r$ most dominant eigenvectors of the symmetric linearized operator. The fourth-order Runge-Kutta scheme with 
time step size $\Delta t=0.4$ (days) is used for the numerical integration of equations~\eqref{eq:full_sys} and~\eqref{eq:reduced_sys}. 
Using the smaller time step $\Delta t=0.2$ (days)  did not change the numerical results.

Figure \ref{Fig:CDV_FTLE_Field} shows the reduced FTLE fields obtained from the one-dimensional reduction (left column) and  two-dimensional reduction (middle column)
as well as the full FTLE (right column). Each row shows a two-dimensional cross section of the phase space.
As in the ABC flow, the reduced FTLE fields agree qualitatively with the fulll FTLE field. 
However, for $r=1$ reduction, some false troughs are observed. 
Away from the false troughs, the one-dimensional reduction estimates the FTLE accurately.
For $r=2$ reduction, there are no false troughs.  

In figure \ref{Fig:CDV_FTLE_His},  the FTLE values 
are plotted as a function of the integration time. These correspond to the point 
$z_0=(1.14,0,0,-0.91,0,0)$, marked by a cross symbol in the inset, which lies on a false trough in the $z_1-z_4$ section. An eigenvalue crossing occurs at approximately $T=9$. Analogues to the ABC flow, we observe that before $T=9$ both one- and two-dimensional OTD reductions capture the most dominant eigenvalues of the Cauchy--Green strain tensor. After $T=9$, the one-dimensional OTD reduction continuously follows the second most dominant eigen-direction of the Cauchy--Green strain tensor. The  inability of a single OTD vector to undergo a dramatic rotation to follow the most dominant eigen-direction of the Cauchy--Green tensor leads to the false trough that is observed in Figure 6(a). The two-dimensional OTD subspace, however, converges to the 
subspace spanned by the two most dominant eigenvectors of the Cauchy--Green strain tensor. 
This convergence is guaranteed by Theorem~\ref{Thm:Alignment} since the second and the third Cauchy--Green eigenvalues $(\lambda_2,\lambda_3)$
do not cross (see figure~\ref{Fig:CDV_FTLE_His}). As a result, the two-dimensional ($r=2$) OTD reduction closely approximates the two larges 
FTLEs for all times. 

\begin{figure}
\centering
\includegraphics[width=.48\textwidth]{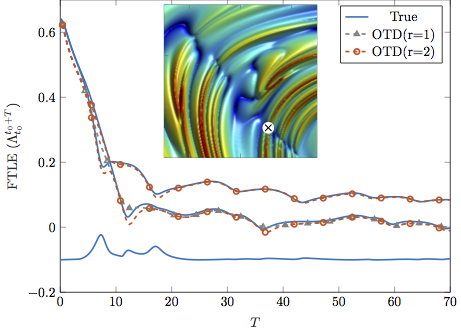}
\caption{Six-dimensional Charney–DeVore model: FTLE versus advection time calculated with full dynamics (solid blue), one-dimensional OTD reduction (dashed gray line with triangle symbols) and two-dimensional OTD reduction (red line with circle symbols) with initial points of   $z_{0,1}=1.14$, $z_{0,4}=-0.91$ and other coordinates being zero. The initial point is shown with the a cross symbol in the $z_1 - z_4$ plane.}
\label{Fig:CDV_FTLE_His}
\end{figure}

The above results demonstrate that, as long as $\lambda_2$ and $\lambda_3$ do not coincide, 
a reduced FTLE with two OTD modes reliably approximates the dominant FTLEs of the full system regardless of the system's dimension $n$.  This amounts to a significant 
reduction in the computational cost of the FTLE evaluations, and the gain in the computational speed is larger as the dimension of the dynamical system increases. 
Investigating the numerical performance of the OTD reduction for infinite-dimensional systems is the subject of future study.

\section{Conclusions}
We have examined the properties of the optimally time-dependent (OTD) modes for general time-dependent dynamical systems. Specifically, we have shown that under mild conditions, related to the spectrum of the Cauchy--Green tensor, the OTD modes converge exponentially fast to the eigendirections of the Cauchy--Green tensor associated with the dominant eigenvalues (largest finite-time Lyapunov exponents). Therefore, the OTD modes can be employed to extract time-dependent subspaces that encode information associated with transient dynamics (finite-time instabilities). 

We have applied the derived result on the formulation of a reduced order algorithm for the computation of the maximum finite-time Lyapunov exponent, a measure that has been used for the quantification of Lagrangian Coherent Structures. We demonstrated the derived results through two specific examples, the three-dimensional ABC flow and the six-dimensional Charney-DeVore model. In both case we thoroughly analyzed the limitations due to dimensionality reduction in combination with eigenvalue crossing. Apart of its value as a computational method for finite-time Lyapunov exponents the presented result paves the way for the development of efficient control and prediction strategies for chaotic dynamical systems exhibiting transient features, such as extreme events and off-equilibrium dynamics.

\subsubsection*{Acknowledgments}
T.P.S. has been supported through the ARO grant 66710-EG-YIP, the AFOSR grant FA9550-16-1-0231, the ONR grant N00014-15-1-2381 and the DARPA grant no. HR0011-14-1-0060. H.B. and M.F. have been supported through the first, second and fourth grants as postdoctoral associates. 

\section*{Appendix A: The OTD modes and the dynamically orthogonal modes}
The set of OTD equations have the same form as the Dynamically Orthogonal
(DO) field equations \cite{sapsis11a, SapsisLermusiaux09}. In fact, because the minimization of the
 function $\mathcal{F}$  is performed only over the rate of change of the basis elements,
$\dot u_i(t)$, (and not the basis elements $u_i(t)$) one can follow exactly the same steps as in Theorem 2.1 in \cite{Babaee} to show that
the evolution equations (\ref{eq:DO_basis}) can be derived for the general
case of a nonlinear operator $\mathcal{N}(U)$. More specifically, by minimization of the
functional:\begin{equation}
\mathcal{F}(\dot u_1,\dot u_2, \dots,\dot u_r) 
 =\sum_{i=1}^{r}\left\Vert \frac{\partial u_{i}\left(t\right)}{\partial
t}-\mathcal{N}({z({t),U(t)},t)}\right\Vert ^{2}
\end{equation}we can obtain the evolution equations.
\begin{equation}
\frac{\partial U}{\partial t}  =\mathcal{N}(U)-UU^{T}\mathcal{N}(U).
\end{equation} For the case of the  DO equations  the
nonlinear operator takes the form\begin{equation}
\mathcal{N}(U)=E^{\omega}[\mathcal{L}(\bar u+\sum^{r} _{i=1}u_iY_i)Y_j]C_{Y_iY_j}^{-1},
\end{equation} 
where we have used the notation in \cite{sapsis11a}, i.e. $\mathcal{L}$ is the
right hand side of the stochastic PDE, $\bar u$ is the trajectory of the mean, $Y_i(t)$ are the stochastic coefficients
and $C_{Y_iY_j}(t)$ is their covariance matrix. The operation $E^\omega$ denotes the average with respect
to an appropriate probability measure. Conversely, one can obtain the OTD mode equations from the DO equations simply by considering their deterministic limit, i.e. by taking the limit $C_{Y_iY_j}(t)\rightarrow0$.

\section*{Appendix B - Proof of Theorem~\ref{thm:dotR}}
We consider the rate of change of the eigenvectors of a general one-parameter family of symmetric matrices,
$G(t) \in \mathbb{R}^{n\times n}$. For such a matrix we have the eigenvalue problem
\begin{equation}
G(t)R(t)=R(t)\Lambda(t)
\end{equation}
where the columns of $R(t)\in \mathbb{R}^{n\times n}$ are the eigenvectors of $G(t)$ and
$\Lambda(t)\in \mathbb{R}^{n\times n}$ is the diagonal matrix of corresponding eigenvalues. 
Note that, since $G$ is symmetric, its eigenvectors are orthogonal and therefore $RR^T=I$.
Differentiating with respect to time, we obtain
\begin{equation*}
\dot{G} R + G \dot{R}= \dot{R}\Lambda + R \dot{\Lambda}. 
\end{equation*}
Multiplying the above equation by $R^T$ from the left yields
\begin{align*}
\dot{\Lambda} &= R^T \dot GR + \Lambda\ R^T\dot{R} - R^T
\dot{R}\Lambda.                    
\end{align*}
Since $R^T R = I$, we have:
\begin{equation*}
\dot{R}^T R + R^T \dot{R} = 0.
\end{equation*}
Let $K = R^T \dot{R}$. From the above relation we observe that $K$ is a skew-symmetric
matrix. Denoting $\widetilde{G} = R^T \dot GR$, we have
\begin{equation*}
  \dot{\Lambda} =  \widetilde{G} + \Lambda K - K \Lambda. 
\end{equation*}
For off-diagonal terms $i \neq j$, we have $\dot{\Lambda}_{ij}=0$. Therefore,
\begin{equation*}
  \widetilde{G}_{ij} = \lambda_j K_{ij} - \lambda_i K_{ij}, \quad i \neq
j.
\end{equation*}
Transposing the above relation
\begin{equation*}
  \widetilde{G}_{ji} = \lambda_i K_{ji} - \lambda_j K_{ji}, \quad i \neq
j.
\end{equation*}
Using the skew-symmetric property of $K$, i.e. $K_{ij}=-K_{ji}$, and summing
up the above two equations yields
\begin{equation}
  K_{ij} = \frac{\widetilde{G}_{ij} + \widetilde{G}_{ji}}{2(\lambda_j - \lambda_i)}
, \quad i \neq j.
\end{equation}
This results in the following closed-form evolution equations for $\Lambda$ and $R$,
\begin{align}
\dot{\Lambda} &=  \mbox{diag}(\widetilde{G}), \\
\dot{R} &= R K.
\end{align}

% --------------------------------------------------------------------------------------------------------------
% --------------------------------------------------------------------------------------------------------------
% --------------------------------------------------------------------------------------------------------------

% \textbf{Results}
% \begin{enumerate}
% \item First example: Double Gyre to convey the concept, limitations of a single OTD evolution, etc.
% \item ABC flow.
% \item Possibly a higher dimension dynamical system.
% \end{enumerate}
\bibliographystyle{plain}  
\bibliography{biblio,library}

\end{document}